\documentclass[journal]{IEEEtran}

\usepackage{tikz}
\graphicspath{{images/}}
\usepackage{amsmath}
\usepackage{latexsym, amssymb}
\usepackage{graphicx}
\usepackage{amsmath, amsbsy}
\usepackage{amsopn, amstext}
\usepackage{ifpdf}
\usepackage{cancel, color}
\usepackage{epstopdf}
\usepackage{amscd}

\def\lplus{\,\rotatebox[]{-90}{$\pm$}\,}

\def\lminus{\vdash}

\def\A{\left< A\right>}
\def\B{\left< B\right>}
\def\C{\left< C\right>}
\def\PR{Pr}
\def\tr{tr}
\DeclareMathOperator{\argmin}{argmin}

\DeclareMathOperator{\id}{id}
\def\cal{\mathcal}

\def\ra{\rightarrow}

\def\a{\alpha}
\def\b{\beta}

\def\0{{\bf 0}}

\newcommand{\R}{{\mathbb R}}
\newcommand{\Q}{{\mathbb Q}}
\newcommand{\N}{{\mathbb N}}

\def\dsum{\mathop{\sum}\limits}

\newtheorem{thm}{Theorem}[section]
\newtheorem{dfn}[thm]{Definition}
\newtheorem{prp}[thm]{Proposition}
\newtheorem{exa}[thm]{Example}
\newtheorem{lem}[thm]{Lemma}
\newtheorem{cor}[thm]{Corollary}
\newtheorem{rem}[thm]{Remark}

\newtheorem{proof}[thm]{Proof}

\begin{document}

\title{Topologies on Quotient Space of Matrices via Semi-tensor Product}

\author{Daizhan Cheng, Zequn Liu
	\thanks{This work is supported partly by the National Natural Science Foundation of China (NSFC) under Grants 61733018, 61333001, and 61773371.}
	\thanks{Daizhan Cheng is with the Key Laboratory of Systems and Control, Academy of Mathematics and Systems Sciences, Chinese Academy of Sciences,
		Beijing 100190, P. R. China (e-mail: dcheng@iss.ac.cn).}
    \thanks{Zequn Liu  is with the Key Laboratory of Systems and Control, Academy of Mathematics and Systems Sciences, Chinese Academy of Sciences,
		Beijing 100190, P. R. China and School of Mathematical Sciences, University of Chinese Academy of Sciences, Beijing 100049, P. R. China (e-mail: liuzequn@amss.ac.cn).}
    \thanks{Submitted to Asian Journal of Control.}
    %\thanks{Tielong Shen and XXX are with XXX (e-mail: XXX).}
    %\thanks{Corresponding author: Daizhan Cheng. Tel.: +86 10 82541232.}
}

%\markboth{IEEE Transactions On Automatic Control, Vol. XX, No. Y, Month 201Z}
%{Cheng \MakeLowercase{\textit{et al.}}: On Dimension-Free Linear System}

\maketitle

\begin{abstract}
An equivalence of matrices via semi-tensor product (STP) is proposed. Using this equivalence, the quotient space is obtained. Parallel and sequential arrangements of the natural projection on different shapes of matrices lead to the product topology and quotient topology respectively. Then the Frobenious inner product of matrices is extended to equivalence classes, which produces a metric on the quotient space. This metric leads to  a metric topology. A comparison for these three topologies is presented. Some topological properties are revealed.
\end{abstract}


\IEEEpeerreviewmaketitle

\section{Preliminaries}%s-1

Semi-tensor product of matrices is  defined as follows \cite{che11}:

\begin{dfn}\label{d1.1} Let $A\in {\cal M}_{m\times n}$, $B\in {\cal M}_{p\times q}$, and $t=n\vee p$ be the least common multiple of $n$ and $p$. Then the semi-tensor product (STP) of $A$ and $B$, denoted by $A\ltimes B$, is defined as
\begin{align}\label{1.1}
A\ltimes B:=\left(A\otimes I_{t/n}\right)\left(B\otimes I_{t/p}\right).
\end{align}
\end{dfn}

As $n=p$ it becomes the conventional matrix product. Hence it is a generalization of conventional matrix product. It has some excellent properties such as
\begin{enumerate}
\item it is applicable to two arbitrary matrices;
\item it has certain commutative properties, called the pseudo-commutativity. Particularly, when the swap matrix is used, more commutative properties are obtained;
\item it can be used to express multi-linear functions/mappings easily;
\item it keeps most properties of conventional matrix product available, etc.
\end{enumerate}

Because of its nice properties, STP has received a  multitude of applications, including (i) logical dynamic systems \cite{che11}, \cite{for13}, \cite{las13}; (ii) systems biology \cite{zha13}, \cite{gao13}; (iii) graph theory and formation control \cite{wan12}, \cite{zha13d}; (iv) circuit design and failure detection \cite{che13}, \cite{li12}, \cite{li12h};
 (v) finite
automata and symbolic dynamics \cite{hoc13}, \cite{xu13}, \cite{zha15}; (vi) coding and cryptography
\cite{zho15}, \cite{zha14d}; (vii) fuzzy control \cite{che12d}, \cite{fen13}; (viii) some engineering applications \cite{liu13c}, \cite{wu15}; and many other topics \cite{che12}, \cite{yan15}, \cite{zho16}, \cite{zou15}; just to name a few.

A recent important development of STP is its application to game theory. Particularly, it has been shown as a powerful tool for modeling, analysis, and control design of finite evolutionary games \cite{che14}, \cite{che15}, \cite{guo13}, \cite{liu16}.

After near 20 years development of STP, now it is time for us to explore its mathematical insides. In this paper, we first reveal an interesting fact that the STP is essentially a product of two classis of matrices. Stimulated by this fact, we formally define an equivalence and then consider the quotient space of matrices under this equivalence. Particularly, we are interested in the topological structure of the quotient space. As the natural projection from matrices to their equivalence classes for different sizes of matrices is arranged in a parallel way or a sequential way, the product topology and quotient topology are obtained respectively.

Then the Frobenious inner product of matrices is modified to an inner product of equivalence classes. Using this inner product, a metric is obtained for quotient space. The  metric topology is also obtained. These three topologies, namely, product topology, quotient topology, and metric topology, are compared, and their relationship is investigated. Finally, some topological properties of the quotient space are presented.

The rest of this paper is organized as follows: Section 2 reviews the equivalence relation of matrices. Using this equivalence, the quotient space is obtained in Section 3. Moreover, two natural topologies, namely, product topology and quotient topology, are also proposed. In Section 4, vector space structure and inner product are derived for quotient space, which make the quotient space an inner product space. In Section 5 a norm and then a  metric are deduced from the inner product. The metric topology is followed. Comparison among these three topologies has also been presented.  Section 6 considers subspaces of the quotient space, orthogonal projection on subspaces is also discussed. Section 7 is a brief conclusion with a conjecture, which is left for further study.

\section{Equivalence of Matrices}

 Observing the STP of matrices carefully, one sees easily that in fact it is a product of two classes: $\A=\{A,A\otimes I_2, A\otimes I_3,\cdots\}$ with $\B=\{B,B\otimes I_2, B\otimes I_3,\cdots\}$. Motivated by this fact, we first define an equivalence on the set of all matrices
$$
{\cal M}:=\bigcup_{m=1}^{\infty}\bigcup_{n=1}^{\infty}{\cal M}_{m\times n}.
$$

\begin{dfn}\label{d2.1} \cite{chepr}
Let $A,~B\in {\cal M}$. $A$ and $B$ are said to be equivalent, denoted by $A\sim B$, if there exist identity matrices $I_{\a}$ and $I_{\b}$ such that
\begin{align}\label{2.1}
A\otimes I_{\a}=B\otimes I_{\b}\quad (:=\Theta).
\end{align}
\end{dfn}
The equivalence class of $A$ is denoted by
\begin{align}\label{2.2}
\A=\{B\;|\;B\sim A\}.
\end{align}

It is necessary to verify the relation defined by (\ref{2.1}) is an equivalence relation, i.e., it is reflexive, symmetric, and transitive \cite{kel75}. The verification is straightforward.

\begin{prp}\label{p2.2}  \cite{chepr} If $A\sim B$, then there exists a $\Lambda$ such that
\begin{align}\label{2.3}
A=\Lambda \otimes I_{\b},\quad B=\Lambda\otimes I_{\a}.
\end{align}
\end{prp}

If there exists an identity matrix $I_k$ such that $A\otimes I_k=B$, then $A$ is said to be a divisor of $B$, and $B$ is a multiple of $A$.

If $\a\wedge\b=1$, that is, $\a$ and $\b$ are co-prime, then the $\Theta$ in (\ref{2.1}) is called the least common multiple of $A$ and $B$, and the $\Lambda$ in (\ref{2.3}) is called the greatest common divisor. $A$ is said to be irreducible, if the only divisor of $A$ is itself. Consider an equivalence class $\A$. Then it is easy to verify that
$$
\A=\{A_1,A_2,A_3,\cdots\},
$$
where $A_i=A_1\otimes I_i$, $i=1,2,3\cdots$, and $A_1$ is irreducible, called the root element of $\A$.

Define a class of matrices as
$$
{\cal M}_{\mu}:=\left\{A\in {\cal M}_{m\times n}\;|\; m/n=\mu\right \}.
$$
Then we have the following partition:
\begin{align}\label{2.4}
{\cal M}=\bigcup_{\mu\in \Q_+}{\cal M}_{\mu},
\end{align}
where $\Q_+$ is the set of positive rational numbers.

Furthermore, set $\mu=\mu_y/\mu_x$ and assume the $\mu_y\wedge \mu_x=1$, where $\wedge$ stands for greatest common devisor.
Define
\begin{align}\label{2.5}
{\cal M}_{\mu}^k:={\cal M}_{k\mu_y\times k\mu_x}.
\end{align}
Then we can further decompose ${\cal M}$ as
\begin{align}\label{2.6}
{\cal M}=\bigcup_{\mu\in \Q_+}\bigcup_{k=1}^{\infty}{\cal M}^k_{\mu}.
\end{align}

We refer to \cite{kel75} for related topological concepts and notations used hereafter.

\begin{dfn}\label{d2.3} The topology ${\cal T}_{{\cal M}}$ on ${\cal M}$ is defined as follows:
\begin{itemize}
\item
Each component ${\cal M}^k_{\mu}$, $\mu\in \Q_+$ and $k\in \N$, is a clopen set.
\item
Within a component ${\cal M}^k_{\mu}$, the natural Euclidean topology of $\R^{k^2\mu_y\mu_x}$ is adopted.
\end{itemize}
\end{dfn}

\begin{rem}\label{r2.4}
\begin{enumerate}
\item
The topological space $\left({\cal M},~{\cal T}_{{\cal M}}\right)$ is a very classical and natural topology, which is commonly used.
\item If $A\in {\cal M}_{\mu}$ and $B\in {\cal M}_{\sigma}$ and $\mu\neq \sigma$, then $A\not\sim B$. Hence, there is no equivalence relation between elements in different ${\cal M}_{\mu}$ components.
\end{enumerate}
\end{rem}

\section{Quotient Space and Its Topologies}

Using the equivalence relation defined in previous section, we can define quotient space as follows:
\begin{align}\label{3.1}
\Sigma:={\cal M}/\sim.
\end{align}
Similarly, we can also define the component quotient space as follows:
\begin{align}\label{3.2}
\Sigma_{\mu}:={\cal M}_{\mu}/\sim,\quad \mu\in \Q_+.
\end{align}

The main purpose of this paper is to build a topology on the quotient space $\Sigma$.
We have
\begin{align}\label{3.3}
\Sigma=\bigcup_{\mu\in \Q_+}\Sigma_{\mu}.
\end{align}
Since there is no equivalence relation cross different components of $\Sigma_{\mu}$'s, it is natural to assume each component $\Sigma_{\mu}$ being clopen. Then we have only to build  topology on each component.

Define a naturel projection $\PR:{\cal M}\ra \Sigma$ as
\begin{align}\label{3.4}
\PR(A):=\A,\quad A\in {\cal M}.
\end{align}
When restrict $\PR$ on each components, we have $\PR:{\cal M}_{\mu}\ra \Sigma_{\mu}$. As aforementioned, we will focus on building a topology on each $\Sigma_{\mu}$.

\subsection{Product Topology}

Consider ${\cal M}_{\mu}$ as a product space:
$$
{\cal M}_{\mu}=\prod_{k=1}^{\infty}{\cal M}_{\mu}^k.
$$
Then the natural mapping $\PR:{\cal M}_{\mu}\ra \Sigma_{\mu}$ can be split into parallel mappings on each component as $\PR\big|_{{\cal M}_{\mu}^k}: {\cal M}_{\mu}^k\ra \Sigma_{\mu}$, $k=1,2,\cdots$. This mapping is described in Fig.1.

%%%%%%%%%%%%%%%%%%%%%%%%%%%%%%%%%%%%%%%%%%%%%%%%%%%%%%%%%%%%%%%%%%%%%%%
\begin{center}
\setlength{\unitlength}{0.6cm}
\begin{picture}(11,6)\thicklines
\put(3,0){\framebox(4,1){$\left(\Sigma_{\mu},~{\cal T}_P\right)$}}
\put(0,3){\framebox(2,1){${\cal M}^1_{\mu}$}}
\put(3,3){\framebox(2,1){${\cal M}^2_{\mu}$}}
\put(7,3){\framebox(2,1){${\cal M}^k_{\mu}$}}
\put(1,3){\vector(1,-1){2}}
\put(4,3){\vector(0,-1){2}}
\put(8,3){\vector(-1,-2){1}}
\put(2.3,3.5){$\times$}
\put(5,3.5){$\times \cdots \times $}
\put(9.5,3.5){$\times \cdots$}
\put(1.6,1.5){Pr}
\put(4.2,1.5){Pr}
\put(7.7,1.5){Pr}
\end{picture}\label{Fig.1}
\end{center}
\vskip 5mm
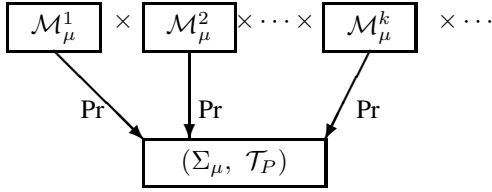
\centerline{Fig.1: Parallel Arrangement of Projection}
\vskip 5mm

To make $\PR\big|_{{\cal M}_{\mu}^k}$ continuous, it is obvious that $\left<O\right>$ is open, if and only if, $O\subset {\cal M}_{\mu}^k$ is open.
Finally, we take the product topology for $\Sigma_{\mu}$.

\begin{dfn}\label{d3.1.1} \cite{kel75} Let $(X_{\lambda},{\cal T}_{\lambda})$, $\lambda\in \Lambda$ be a set of topological spaces. The product set
$$
X:=\prod_{\lambda\in \Lambda}X_{\lambda}
$$
 with product topology, is called a product topological space, where the product topology is generated by
$$
S:=\{ O \;|\; ~\mbox{there is a}~~\lambda, ~~\mbox{such that}~~O\in {\cal T}_{\lambda}\}
$$
as the topological subbase.
\end{dfn}

Now consider $\Sigma_{\mu}$. Its product topology is generated as follows:
\begin{itemize}
\item Step 1: Topological subbase $S$:
$$
S_k:=\left\{\left<O_k\right>\;|\;O_k\in {\cal T}_{{\cal M}_{\mu}^k}\right\},
$$
where and hereafter, we use the notation
$$
\left<O\right>:=\{\A\;|\;A\in O\}.
$$

$$
S:=\bigcup_{k=1}^{\infty}S_k.
$$
\item Step 2: Topological base $B$:
$$
B:=\left\{\bigcap_{i=1}^t O_i\;|\; O_i\in S, t<\infty\right\}.
$$
(That is, $B$ is the set of finite intersections of elements in $S$.)

\item Step 3: Product topology on $\Sigma_{\mu}$, denoted by ${\cal T}_P$:
$$
{\cal T}_P=\left\{\bigcup_{\lambda\in \Lambda}O_{\lambda}\;|\;O_{\lambda}\in B\right\}.
$$
(That is,  ${\cal T}_P$ is the set of arbitrary union of elements in $B$.)
\end{itemize}

\begin{thm}\label{t3.1.2} $\left(\Sigma_{\mu}, {\cal T}_P\right)$ is a second countable Hausdorff space. \footnote{(1) A topological space is second countable if it has a countable topological basis. (2) A topological space is Hausdorff if any two points have their open neighborhoods, which are separated.}
\end{thm}
\begin{proof} Choose all $k$-dimensional open balls $O_k$  with rational numbers as their center coordinates and radius to form $\langle O_k\rangle\in S_k$. Denote the set of such balls by $O^k$. Then $|O^k|=\aleph_0$. Set $O:=\bigcup_{k=1}^{\infty} O^k$, since a countable union of countable sets is countable, $|O|=\aleph_0$.
Finally, choose $B^O$ be the set of finite intersections of elements in $O$, then it is clear that (i) $|B^O|=\aleph_0$, i.e., it is countable; (ii) $B^O$ is a topological base of ${\cal T}_P$. Hence, $\left(\Sigma_{\mu}, {\cal T}_P\right)$ is second countable.

To prove it is Hausdorff, let $\A,~\B\in \Sigma_{\mu}$, $\A\neq \B$. Assume their root elements are $A_1\in {\cal M}_{\mu}^p$ and $B_1  \in {\cal M}_{\mu}^q$, and $p\vee q=t$. Then $A_{t/p}=A_1\otimes I_{t/p}\in {\cal M}_{\mu}^t$ and   $B_{t/q}=B_1\otimes I_{t/q}\in {\cal M}_{\mu}^t$. Since
${\cal M}_{\mu}^t$ is Hausdorff, there exist open sets $O_p\cap O_q=\emptyset$, such that $A_{t/p}\in O_p$ and $B_{t/q}\in O_q$. Then we have
(i) $\left<O_p\right>$ and $\left<O_q\right>$ are open, that is,  $\left<O_p\right>,~\left<O_q\right> \in {\cal T}_P$, (ii) $\left<O_p\right> \bigcap \left<O_q\right>=\emptyset$, (iii) $\A\in \left<O_p\right>$ and $\B\in \left<O_q\right>$. Hence, $\left(\Sigma_{\mu}, {\cal T}_P\right)$ is Hausdorff.

\end{proof}

\subsection{Quotient Topology}

\begin{dfn}\label{d3.2.1} \cite{kel75} Let $(X,{\cal T}_X)$ and  $(Y,{\cal T}_Y)$ be two topological spaces, and $\PR:X\ra Y$. Assume
\begin{enumerate}
\item $\PR$ is a surjective (i.e., onto) mapping;
\item $O\in {\cal T}_Y$, if and only if, $\PR^{-1}(O)\in {\cal T}_X$.
\end{enumerate}
Then $Y$ is called a quotient space with quotient topology ${\cal T}_Y$.
\end{dfn}

Now consider $\PR: {\cal M}_{\mu}\ra \Sigma_{\mu}$. It is obvious that $\PR$ is surjective. To get the quotient topology we construct a sequential projection, which is described in Fig.2.

%%%%%%%%%%%%%%%%%%%%%%%%%%%%%%%%%%%%%%%%%%%%%%%%%%%%%%%%%%%%%%%%%%%%%%
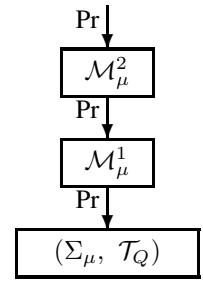
\begin{center}
\setlength{\unitlength}{0.6cm}
\begin{picture}(5,8)\thicklines
\put(0,0){\framebox(4,1){$\left(\Sigma_{\mu},~{\cal T}_Q\right)$}}
\put(1,2){\framebox(2,1){${\cal M}^1_{\mu}$}}
\put(1,4){\framebox(2,1){${\cal M}^2_{\mu}$}}
\put(2,6){\vector(0,-1){1}}
\put(2,4){\vector(0,-1){1}}
\put(2,2){\vector(0,-1){1}}
\put(1.3,5.5){Pr}
\put(1.3,3.5){Pr}
\put(1.3,1.5){Pr}
\end{picture}\label{Fig.2}
\end{center}
\vskip 5mm
\centerline{Fig.2: Sequential Arrangement of Projection}
%%%%%%%%%%%%%%%%%%%%%%%%%%%%%%%%%%%%%%%%%%%%%%%%%%%%%%%%%%%%%%%%%%%%%%%
\vskip 5mm

Denote the quotient topology on $\Sigma_{\mu}$ by ${\cal T}_Q$. By definition ${\cal T}_Q$ has the following structure.

\begin{prp}\label{p3.2.2} A set $U\in {\cal T}_Q$, if and only if,
\begin{align}\label{3.2.1}
\PR^{-1}(U)\bigcap{\cal M}^k_{\mu}\in {\cal T}_{{\cal M}^k_{\mu}},\quad \forall k\in \N.
\end{align}
\end{prp}

\begin{cor}\label{c3.2.3}  ${\cal T}_Q\neq {\cal T}_P$.
\end{cor}
\begin{proof}
Let $O\in {\cal T}_{{\cal M}^k_{\mu}}$. By definition $\left<O\right>\in {\cal T}_P$. But
for $s>k$, one sees easily that
$$
\PR^{-1}(\left<O\right>)\bigcap{\cal M}^s_{\mu}\notin {\cal T}_{{\cal M}^s_{\mu}},
$$
because $\PR^{-1}(\left<O\right>)\bigcap{\cal M}^s_{\mu}$ is a subset of a lower dimensional subspace of ${\cal M}^s_{\mu}$. It can not be open. Hence,
$\left<O\right>\notin {\cal T}_Q$.
\end{proof}

\section{Vector Space Structure on Quotient Space}

Let $A=(a_{i,j}),~B=(b_{i,j}) \in {\cal M}_{m\times n}$. Recall that the Frobenius inner product is defined as \cite{hor85}
\begin{align}\label{4.1}
(A|~B)_F:=\dsum_{i=1}^m\dsum_{j=1}^n a_{i,j}b_{i,j}.
\end{align}
Correspondingly, the Frobenius norm is defined as
\begin{align}\label{4.2}
\|A\|_F:=\sqrt{(A|~A)_F}.
\end{align}

The following lemma comes from a straightforward computation.

\begin{lem}\label{l4.1} Let $A,~B\in {\cal M}_{m\times n}$. Then
\begin{align}\label{4.3}
\left(A\otimes I_k\;|\;B\otimes I_k\right)_F=k(A\;|\;B)_F.
\end{align}
\end{lem}

Taking Lemma \ref{l4.1} into consideration, we propose the following definition, which is independent of the size of matrices.

\begin{dfn}\label{d4.2} Let $A,~B\in {\cal M}_{\mu}$, where $A\in {\cal M}_{\mu}^{\a}$ and  $B\in {\cal M}_{\mu}^{\b}$. Then the weighted inner product of $A,~B$ is defined as
\begin{align}\label{4.4}
(A\;|\;B)_W:=\frac{1}{t}\left(A\otimes I_{t/\a}\;|\;~B\otimes I_{t/\b}\right)_F,
\end{align}
where $t=lcm(\a,~\b)$ is the least common multiple of $\a$ and $\b$.
\end{dfn}

Then the norm of $A\in {\cal M}_{\mu}$ is defined naturally as:
\begin{align}\label{4.401}
\|A\|:=\sqrt{(A\;|\;A)_{W}}.
\end{align}

Using Lemma \ref{l4.1} and Definition \ref{d4.2}, we have the following property.

\begin{prp}\label{p4.3} Let $A,~B\in {\cal M}_{\mu}$, if $A$ and $B$ are orthogonal, i.e., $(A\;|\;~B)_F=0$, then
$A\otimes I_{\xi}$ and $B\otimes I_{\xi}$ are also orthogonal.
\end{prp}

Now we are ready to define an inner product on $\Sigma_{\mu}$.

\begin{dfn}\label{d4.4} Let $\A,~\B\in \Sigma_{\mu}$. Their inner product is defined as
\begin{align}\label{4.5}
(\A\;|\;~\B):=(A\;|\;~B)_W.
\end{align}
\end{dfn}

The following proposition shows that (\ref{4.5}) is well defined.

\begin{prp}\label{p4.5} Definition \ref{d4.4} is well defined. That is, (\ref{4.5}) is independent of the choice of representatives $A$ and $B$.
\end{prp}

\begin{proof} Assume $A_1\in \A$ and $B_1\in \B$ are irreducible. Then it is enough to prove that
\begin{align}\label{4.6}
(A\;|\;~B)_W=(A_1\;|\;~B_1)_W,\quad A\in \A,~B\in \B.
\end{align}
Assume $A_1\in {\cal M}_{\mu}^{\a}$ and $B_1\in {\cal M}_{\mu}^{\b}$.   Let
$$
\begin{array}{l}
A=A_1\otimes I_{\xi}\in {\cal M}_{\mu}^{\a \xi}\\
B=B_1\otimes I_{\eta}\in {\cal M}_{\mu}^{\b \eta}.\\
\end{array}
$$
Denote  $t=\a\vee \b$, $s=\a\xi\vee \b\eta$, and $s=t\ell$.
Using (\ref{4.4}), we have
$$
\begin{array}{ccl}
(A\;|\;~B)_W&=&\frac{1}{s}\left(A\otimes I_{\frac{s}{\a\xi}}\;|\; B\otimes I_{\frac{s}{\b\eta}}\right)_F\\
~&=&\frac{1}{s}\left(A_1\otimes I_{\frac{s}{\a}}\;|\; B_1\otimes I_{\frac{s}{\b}}\right)_F\\
~&=&\frac{1}{t\ell}\left(A_1\otimes I_{\frac{t}{\a}}\otimes I_{\ell}\;|\; B_1\otimes I_{\frac{t}{\b}}\otimes I_{\ell}\right)_F\\
~&=&\frac{1}{t}\left(A_1\otimes I_{\frac{t}{\a}}\;|\; B_1\otimes I_{\frac{t}{\b}}\right)_F\\
~&=&(A_1\;|\;~B_1)_W.
\end{array}
$$
\end{proof}

\begin{dfn}\label{dmt.1.6} \cite{tay80} A real vector space $X$ is an inner-product space, if there is a mapping $X\times X\ra \R$, denoted by $(x|~y)$, satisfying
\begin{enumerate}
\item
$$
(x+y\;|\;z)=(x\;|\;z)+(y\;|\;z),\quad x,y,z\in X.
$$
\item
$$
(x\;|\;y)=(y\;|\;x).
$$
\item
$$
(ax\;|\;y)=a(x\;|\;y),\quad a\in \R.
$$
\item
$$
(x\;|\;x)\geq 0, ~\mbox{and}~ (x\;|\;x)\neq 0 ~\mbox{if}~ x\neq 0.
$$
\end{enumerate}
\end{dfn}

To pose a vector space structure on $\Sigma_{\mu}$, we first define the addition and scalar product on it.

\begin{dfn}\label{d4.601} \cite{chepr} Let $A\in {\cal M}_{\mu}^p$ and $B\in {\cal M}_{\mu}^q$. Then
\begin{enumerate}
\item Addition:
\begin{align}\label{4.601}
A\lplus B:=\left(A\otimes I_{t/p}\right) + \left(B\otimes I_{t/q}\right),
\end{align}
where $t=p\vee q$.
\item Subtraction:
\begin{align}\label{4.602}
A\lminus B:=A\lplus (-B).
\end{align}
\end{enumerate}
\end{dfn}

\begin{dfn}\label{d4.602} \cite{chepr} Let $\A,~\B\in \Sigma_{\mu}$. Then
\begin{enumerate}
\item Addition:
\begin{align}\label{4.603}
\A\lplus \B:=\left<A\lplus B   \right>.
\end{align}
\item Subtraction:
\begin{align}\label{4.604}
\A\lminus \B:=\left<A\lminus B   \right>.
\end{align}
\item Scalar product:
\begin{align}\label{4.605}
r\A:=\left<rA\right>,\quad r\in \R.
\end{align}
\end{enumerate}
\end{dfn}

\begin{prp}\label{p4.603} \cite{chepr}
\begin{enumerate}
\item (\ref{4.603})-(\ref{4.605}) are properly defined. That is, they are independent of the choice of representatives.
\item $\Sigma_{\mu}$ with addition (subtraction) and scalar product, defined in Definition \ref{d4.602}, is a vector space.
\end{enumerate}
\end{prp}

Furthermore,  it is easy to verify the following result.

\begin{thm}\label{t4.7} The vector space $(\Sigma_{\mu},\lplus)$ with the inner product defined by (\ref{4.5}) is an inner product space.
\end{thm}

Then the norm of $\A\in \Sigma_{\mu}$ is defined naturally as:
\begin{align}\label{4.7}
\|\A\|:=\sqrt{(\A\;|\;\A)}.
\end{align}

The following is some standard results for inner product space.

\begin{thm}\label{t4.8} Assume $\A,~\B\in \Sigma_{\mu}$. Then we have the following
\begin{enumerate}
\item (Schwarz Inequality)
\begin{align}\label{4.8}
|(\A\;|\;~\B)|\leq \|\A\|\|\B\|;
\end{align}
\item (Triangular Inequality)
\begin{align}\label{4.9}
\|\A\lplus \B\| \leq \|\A\|+\|\B\|;
\end{align}
\item (Parallelogram Law)
\begin{align}\label{4.10}
\begin{array}{l}
\|\A\lplus \B\|^2+\|\A\lminus \B\|^2 \\
~~= 2\|\A\|^2+2\|\B\|^2.
\end{array}
\end{align}
\end{enumerate}
\end{thm}

Note that the above properties show that $\Sigma_{\mu}$ is a normed space. Unfortunately, it is not difficult to show that $\Sigma_{\mu}$ is not a Hilbert space.

\section{ Matric and  Matric Topology}%3.5
Using the norm defined in previous section, it is ready to verify that $\Sigma_{\mu}$ is a metric space. We state it as a theorem.

\begin{thm}\label{t5.1} $\Sigma_{\mu}$ with distance
\begin{align}\label{5.1}
d(\A,\B):=\|\A\lminus \B\|, \quad \A,~\B\in \Sigma_{\mu}
\end{align}
is a metric space.
\end{thm}

\begin{thm}\label{t5.2} Consider $\Sigma_{\mu}$. The metric topology determined by the distance $d$ is denoted by ${\cal T}_d$. Then
\begin{align}\label{5.2}
{\cal T}_d\subset {\cal T}_Q\subset{\cal T}_P.
\end{align}
\end{thm}

\begin{proof}
First, we prove the first half inclusion. Assume $V\in {\cal T}_d$, then for each $\left<p\right>\in V$ there exists an $\epsilon>0$, such that a ball $B_{\epsilon}( \left<p\right>)\subset V$. Because $d(\A,\B)=d(A,B)$, we have
$$
B_{\epsilon}(p)\subset \PR^{-1}\big|_{{\cal M}_{\mu}^s}(V),\quad p\in \left<p\right>\bigcap {\cal M}_{\mu}^s.
$$
That is,  $\PR^{-1}\big|_{{\cal M}_{\mu}^s}(V)$ is open in ${\cal M}_{\mu}^s$. By definition of quotient topology, $V$ is an open set in ${\cal T}_Q$. It follows that ${\cal T}_d\subset {\cal T}_Q$.

Next, we prove the second half inclusion. Assume $V\in {\cal T}_Q$, then for each $\left<p\right>\in V$ we can find $p_0\in \left<p\right>\bigcap {\cal M}_{\mu}^s$. By definition of ${\cal T}_Q$, $V_s:=\PR^{-1}(V)\bigcap {\cal M}_{\mu}^s$ is open and $p_0\in V_s$. Then
$$
\left<V_s\right>\subset V.
$$
By definition of  ${\cal T}_P$ we have
$$
\left<p\right>\in \left<V_s\right>\in {\cal T}_P,
$$
which means $V\in  {\cal T}_P$. \footnote{An alternative more holistic proof is provided by an anonymous reviewer as follows: for all $V\in{\cal T}_Q$, one has $Pr^{-1}(V)\cap {\cal M}_{\mu}^k\in{\cal T}_{{\cal M}_{\mu}^k}$ for all positive integers $k$. Note that $Pr^{-1}(V)=\bigcup_{k=1}^{\infty}(Pr^{-1}(V)\cap {\cal M}_{\mu}^k)$, then by subjectivity of $Pr$, $V=Pr(Pr^{-1}(V))=\bigcup_{k=1}^{\infty}Pr(Pr^{-1}(V))\cap
{\cal M}_{\mu}^k)\in{\cal T}_P$.}

\end{proof}

Next, we consider the upper-bounded subspace
$$
\Sigma_{\mu}^{[\cdot,k]}:=\bigcup_{i=1}^k \Sigma_{\mu}^i=\bigcup_{i=1}^k {\cal M}_{\mu}^i/\sim.
$$
 Then we have the following result:

\begin{thm}\label{t5.3} Consider $\Sigma_{\mu}^{[\cdot,k]}$.  Then
\begin{align}\label{5.3}
{\cal T}_d\big|_{\Sigma_{\mu}^{[\cdot,k]}}= {\cal T}_Q\big|_{\Sigma_{\mu}^{[\cdot,k]}}\subset {\cal T}_P\big|_{\Sigma_{\mu}^{[\cdot,k]}}.
\end{align}
\end{thm}

\begin{proof} Define
 \begin{align}\label{5.4}
\left<{\cal T}_{{\cal M}_{\mu}^k}\right>:=\left\{\left<O\right>\;|\;O\in {\cal T}_{{\cal M}_{\mu}^k}\right\}.
\end{align}
Note that ${\cal T}_{{\cal M}_{\mu}^k}$ is the topology on ${\cal M}_{\mu}^k$, which is the standard Euclidean topology on ${\cal M}_{\mu}^k\approxeq \R^{k^2\mu_y\mu_x}$. Hence $O$ is a standard Euclidean open set.

By definition of ${\cal T}_P$ it is clear that
 \begin{align}\label{5.5}
\left<{\cal T}_{{\cal M}_{\mu}^k}\right>\subset {\cal T}_P\big|_{\Sigma_{\mu}^{[\cdot,k]}}.
\end{align}

Next, we claim that
 \begin{align}\label{5.6}
{\cal T}_Q\big|_{\Sigma_{\mu}^{[\cdot,k]}}=
\left<{\cal T}_{{\cal M}_{\mu}^k}\right>.
\end{align}
Assume $O\in {\cal T}_{{\cal M}_{\mu}^k}$, and $p\in O$. Then there exists an $\epsilon>0$ such that $B_{\epsilon}(p)\subset O$. Construct
$$
B_{\epsilon}(\left<p\right>)\subset \langle O\rangle.
$$
Then it is clear that
$$
\PR^{-1}\big|_{{\cal M}_{\mu}^k} B_{\epsilon}(\left<p\right>)=B_{\epsilon}(p),
$$
which is open in ${\cal M}_{\mu}^k$.

As for $s<k$. Since ${\cal M}_{\mu}^s$ is a subspace of ${\cal M}_{\mu}^k$,
$$
\PR^{-1}\big|_{{\cal M}_{\mu}^s} B_{\epsilon}(\left<p\right>)=B_{\epsilon}(p)\bigcap {\cal M}_{\mu}^s,
$$
which is also open in ${\cal M}_{\mu}^s$.
By definition of quotient space,
$$
B_{\epsilon}(\left<p\right>)\in {\cal T}_Q\big|_{\Sigma_{\mu}^{[\cdot,k]}}.
$$
It follows immediately that
$$
\left<{\cal T}_{{\cal M}_{\mu}^k}\right>\subset {\cal T}_Q\big|_{\Sigma_{\mu}^{[\cdot,k]}}.
$$
Conversely, assume $O\notin {\cal T}_Q({\cal M}_{\mu}^k)$, and $\PR(O)\in {\cal T}_Q\big|_{\Sigma_{\mu}^{[\cdot,k]}}$. Then
$$
\PR^{-1}\big|_{{\cal M}_{\mu}^k}\left( \PR(O) \right)=O \notin  {\cal T}_Q({\cal M}_{\mu}^k),
$$
which contradicts to the definition of quotient space. Hence,
$$
\left<{\cal T}({\cal M}_{\mu}^k)\right>\supset {\cal T}_Q\big|_{\Sigma_{\mu}^{[\cdot,k]}}.
$$
 (\ref{5.6}) is proved.

Similarly, we can also prove
\begin{align}\label{5.7}
{\cal T}_d\big|_{\Sigma_{\mu}^{[\cdot,k]}}=
\left<{\cal T}({\cal M}_{\mu}^k)\right>.
\end{align}
(\ref{5.5})-(\ref{5.7}) lead to (\ref{5.3}).
\end{proof}

\vskip 5mm

The following is a {\bf conjecture}:
 \begin{align}\label{5.8}
{\cal T}_d={\cal T}_Q.
\end{align}

\vskip 5mm

\begin{dfn}\label{d5.4} \cite{dug66}
\begin{enumerate}
\item A topological space is regular (or $T_3$) if for each closed set $X$ and $x\not\in X$ there exist open neighborhoods $U_x$ of $x$ and $U_X$ of $X$, such that $U_x\cap U_X=\emptyset$.
\item A topological space is normal (or $T_4$) if for each pair of closed sets $X$ and $Y$ with $X\cap Y=\emptyset$, there exist open neighborhoods $U_X$ of $X$ and $U_Y$ of $Y$, such that $U_X\cap U_Y=\emptyset$.
\end{enumerate}
\end{dfn}

Note that $\left(\Sigma_{\mu},~{\cal T}_{d}\right)$ is a metric space. A metric space is $T_4$. Moreover,
$$
T_4\Rightarrow T_3\Rightarrow T_2~(\mbox{Hausdoff}).
$$
Using  Theorem \ref{t5.2}, we have the following result.

\begin{cor}\label{c5.4}
\begin{itemize}
\item The topological space $\left(\Sigma_{\mu},{\cal T}\right)$ is a Hausdorff space, where
${\cal T}$ can be any one of ${\cal T}_{Q}$, ${\cal T}_{P}$, or ${\cal T}_{d}$.
\item The topological space $\left(\Sigma_{\mu},{\cal T}\right)$ is both regular and normal, where
${\cal T}$ can be either ${\cal T}_{Q}$ or ${\cal T}_{d}$.
\end{itemize}
\end{cor}

Finally, we show some properties of $\Sigma_{\mu}$.

\begin{prp}\label{p5.5}
$\Sigma_{\mu}$ is convex. Hence it is arcwise connected.
\end{prp}
\begin{proof} Assume $\A,~\B\in \Sigma_{\mu}$. Then it is clear that
$$
\lambda\A\lplus (1-\lambda)\B=\left<\lambda A\lplus (1-\lambda) B\right>\in \Sigma_{\mu},\quad \lambda\in [0,1].
$$
So $\Sigma_{\mu}$ is convex. Let $\lambda$ go from $1$ to $0$, we have a path connecting $\A$ and $\B$.
\end{proof}

\begin{prp}\label{p5.6}
$\Sigma_{\mu}$ and $\Sigma_{1/\mu}$ are isometric spaces.
\end{prp}

\begin{proof} Consider an isomorphic mapping $\varphi: \Sigma_{\mu}\ra\Sigma_{1/\mu}$, defined by transpose:
$$
\varphi(\A):= \left<A^T\right>.
$$
Then it is obvious that
$$
d(\A,~\B)=d\left(\left<A^T\right>, \left<B^T\right>\right).
$$
Hence the transpose is an isometry. Moreover, $\varphi$ is periodic. That is,
$$
\varphi^2=\id.
$$
\end{proof}

\begin{rem}\label{r5.6}
Since ${\cal T}_d$ is expected to be the same as ${\cal T}_Q$, we do not try to distinct them. They are mainly represent the finite union of different dimensional Euclidian spaces. They are particularly useful for investigating cross-dimensional dynamic systems. While ${\cal T}_P$ represents the infinite union of different dimensional Euclidian spaces. It has more mathematical inside.
\end{rem}

\section{Subspaces of $\Sigma_{\mu}$}%3.6

Consider the $k$-upper bounded subspace $\Sigma_{\mu}^{[\cdot,k]}\subset \Sigma_{\mu}$. We have

\begin{prp}\label{p6.1} $\Sigma_{\mu}^{[\cdot,k]}$ is a Hilbert space.
\end{prp}

\begin{proof} Since $\Sigma_{\mu}^{[\cdot,k]}$ is a finite dimensional vector space and any finite dimensional inner product space is a Hilbert space \cite{die69}, the conclusion follows.
\end{proof}

\begin{prp}\label{pmt.3.2} \cite{die69} Let $E$ be an inner product space, $\{0\}\neq F\subset E$ be a Hilbert subspace.
\begin{enumerate}
\item For each $x\in E$ there exists a unique $y:=P_F(x)\in F$, called the projection of $x$ on $F$, such that
\begin{align}\label{6.1}
\|x-y\|=\min_{z\in F}\|x-z\|.
\end{align}
\item
\begin{align}\label{6.2}
F^{\perp}:=P_F^{-1}\{0\}
\end{align}
is the subspace orthogonal to $F$.
\item
\begin{align}\label{6.3}
E=F\oplus F^{\perp},
\end{align}
where $\oplus$ stands for orthogonal sum.
\end{enumerate}
\end{prp}

Using above proposition, we consider the projection: $P_F:\Sigma_{\mu}\ra \Sigma_{\mu}^{[\cdot,\a]}$. Let $\A\in \Sigma_{\mu}^{\b}$. Assume $\left<X\right>\in \Sigma_{\mu}^{\a}$, $t=\a\vee \b$. Then the norm of $\A\lminus \left<X\right>$ is:
\begin{align}\label{6.4}
\left\|\A\lminus \left<X\right>\right\|=\frac{1}{\sqrt{t}}\left\|A\otimes I_{t/\b}-X\otimes I_{t/\a}\right\|_F.
\end{align}
Set $p=\mu_y$ $q=\mu_x$, and $k:=t/\a$. We split $A$ as
\begin{align}\label{6.401}
A\otimes I_{t/\b}=\begin{bmatrix}
A_{1,1}&A_{1,2}&\cdots&A_{1,q\a}\\
A_{2,1}&A_{2,2}&\cdots&A_{2,q\a}\\
\vdots&~&~&~\\
A_{p\a,1}&A_{p\a,2}&\cdots&A_{p\a,q\a}\\
\end{bmatrix},
\end{align}
where $A_{i,j}\in {\cal M}_{k\times k}$, $i=1,\cdots,p\a;~j=1,\cdots,q\a$.
Set
\begin{align}\label{6.5}
C:=\argmin_{X\in {\cal M}_{\mu}^{\a}}\left\|A\otimes I_{t/\b}-X\otimes I_{t/\a}\right\|.
\end{align}
Then the projection $P_F:\Sigma_{\mu}\ra \Sigma_{\mu}^{\a}$ is defined by
\begin{align}\label{6.501}
P_F(\A):=\C,\quad \A\in \Sigma_{\mu}^{\b},\; \C\in \Sigma_{\mu}^{\a}.
\end{align}

It is easy to verify the following result:
\begin{prp}\label{p6.3}
\begin{enumerate}
\item Assume $P_F(\A)=\C$, where $A=(A_{i,j})$ is defined by (\ref{6.401}) and $C=(c_{i,j})$ is defined by (\ref{6.5}). Then
\begin{align}\label{6.6}
c_{i,j}=\frac{1}{k}\tr(A_{i,j}),\quad i=1,\cdots,p\a;~j=1,\cdots,q\a,
\end{align}
where $\tr(A)$ is the trace of $A$.
\item The following orthogonality holds:
\begin{align}\label{6.7}
P_F(\A)\perp \A- P_F(\A).
\end{align}
\end{enumerate}
\end{prp}

We give an example to depict the projection.

\begin{exa}\label{e6.4}

Given
$$
A=\begin{bmatrix}
1&2&-3&0&2&1\\
2&1&-2&-1&1&0\\
0&-1&-1&3&1&-2
\end{bmatrix}\in \Sigma_{0.5}^3.
$$
We consider the projection of $\A$ onto $\Sigma_{0.5}^{[\cdot,2]}$. Denote  $t=2\vee 3=6$. Using formulas (\ref{6.6})-(\ref{6.7}), we have
$$
P_F(\A)=\left<\begin{bmatrix}
1&0&1/3&0\\
0&-1/3&0&-1\\
\end{bmatrix}\right>.
$$
Then we have
$$
\left<E\right>=\A \lminus P_F(\A),
$$
where
$$
\begin{array}{l}
E=\\
\left[
\begin{array}{cccccccccccc}
0&0&2&0&-3&0&-\frac{1}{3}&0&2&0&1&0\\
0&0&0&2&0&-3&0&-\frac{1}{3}&0&2&0&1\\
2&0&0&0&-2&0&-1&0&\frac{2}{3}&0&0&0\\
0&2&0&\frac{4}{3}&0&-2&0&-1&0&2&0&0\\
0&0&-1&0&-\frac{2}{3}&0&3&0&1&0&-1&0\\
0&0&0&-1&0&-\frac{2}{3}&0&3&0&1&0&-1\\
\end{array}\right].
\end{array}
$$

It is easy to verify that $\left< E \right>$ and $\A$ are mutually orthogonal.

\end{exa}

We also have $\Sigma_{\mu}^{[k,\cdot]}$ and $\Sigma_{\mu}^{[\a,\b]}$ (where $\a|\b$) as metric subspaces of $\Sigma_{\mu}$.

 Finally, we would like to point out that since $\Sigma_{\mu}$ is an infinity dimensional vector space, it is possible that $\Sigma_{\mu}$ is isometric to its proper subspace. For instance, consider the following example.

 \begin{exa}\label{e6.5} Consider a mapping $\varphi: {\cal M}_{\mu}\ra {\cal M}_{\mu}^{[k,\cdot]}$ defined by
$A\mapsto A\otimes I_k$. It is clear that this mapping satisfies
$$
 \| A \lminus B\|_{W}=\|\varphi(A) \lminus \varphi(B)\|_{W},\quad A, B\in {\cal M}_{\mu}.
 $$
 That is, ${\cal M}_{\mu}$ can be isometrically embedded into its proper subspace ${\cal M}_{\mu}^{[k,\cdot]}$. Define $\varphi:\Sigma_{\mu}\ra \Sigma{\mu}^{[k,\cdot]}$ by $\varphi(\A)=\left<A\otimes I_k\right>$. Then one sees also that $\Sigma_{\mu}$ is isometrically embedded into its proper subspace $\Sigma{\mu}^{[k,\cdot]}$.
 \end{exa}

\section{Conclusion}

It has been revealed that the STP is essentially a product of two equivalence classes. Motivated by this, the equivalence relation has been discussed carefully. Based on this equivalence, the quotient space is constructed. Then the topological structure has been investigated in detail. Three different topologies have been built. First two topologies are (i) product topology ${\cal T}_P$, and (ii) quotient topology ${\cal T}_Q$, which are natural. Then a vector space structure and an inner product have been proposed for the quotient space. Using this inner product, norm and distance/metric are also defined.
Then the metric topology ${\cal T}_d$ is obtained. Basic properties of all three topologies have been discussed. Finally, a comparison among three topologies is presented.

A conjecture is: ${\cal T}_Q={\cal T}_d$, which left for further study.

\end{document}